\documentclass[11pt, reqno]{amsart}
\usepackage{amsmath, amsthm, amscd, amsfonts, amssymb, graphicx, color}
\usepackage[bookmarksnumbered, colorlinks, plainpages]{hyperref}
\hypersetup{colorlinks=true,linkcolor=red, anchorcolor=green, citecolor=cyan, urlcolor=red, filecolor=magenta, pdftoolbar=true}

\DeclareMathSymbol{\subsetneqq}{\mathbin}{AMSb}{36}

\newcommand{\R}{\mathbb{R}}

\newcommand{\beq}{\begin{eqnarray}}
\newcommand{\eeq}{\end{eqnarray}}
\newcommand{\bq}{\begin{equation}}
\newcommand{\eq}{\end{equation}}
\newcommand{\beqn}{\begin{eqnarray*}}
\newcommand{\eeqn}{\end{eqnarray*}}
\newcommand{\bex}{\begin{exo}}
\newcommand{\eex}{\end{exo}}
\newcommand{\ben}{\begin{enumerate}}
\newcommand{\een}{\end{enumerate}}

\newtheorem{th1}{{\bf Theorem}}[section]
\newtheorem{thm}[th1]{{\bf Theorem}}
\newtheorem{lem}[th1]{{\bf Lemma}}
\newtheorem{prop}[th1]{{\bf Proposition}}
\newtheorem{cor}[th1]{{\bf Corollary}}
\newtheorem{rem}[th1]{\bf Remark}
\newtheorem{rems}[th1]{\bf Remarks}
\newtheorem{defi}[th1]{\bf Definition}

\author[S. Kim and T. Saanouni]{Seongyeon Kim and Tarek Saanouni}
\address{School of Mathematics, Korea Institute for Advanced Study, Seoul 02455, Republic of Korea.}

\address{Departement of Mathematics, College of Science and Arts in Uglat Asugour, Qassim University, Buraydah, Kingdom of Saudi Arabia.}

\email{\it synkim@kias.re.kr}
\email{\sl tarek.saanouni@ipeiem.rnu.tn}
\email{\sl t.saanouni@qu.edu.sa}


\title[Hartree INLS]{Energy-critical inhomogeneous generalized Hartree equation with inverse square potential}

\date{\today}
\begin{document}
\begin{abstract}
This work studies the Cauchy problem for the energy-critical inhomogeneous Hartree equation with inverse square potential
$$i\partial_t u-\mathcal K_\lambda u=\pm |x|^{-\tau}|u|^{p-2}(I_\alpha *|\cdot|^{-\tau}|u|^p)u, \quad \mathcal K_\lambda=-\Delta+\frac\lambda{|x|^2}$$
in the energy space $H_\lambda^1:=\{f\in L^2,\quad\sqrt{\mathcal{K}_\lambda}f\in L^2\}$. In this paper, we develop a well-posedness theory and investigate the blow-up of solutions in $H_\lambda^1$. Furthermore we present a dichotomy between energy bounded and non-global existence of solutions under the ground state threshold.
To this end, we use Caffarelli-Kohn-Nirenberg weighted interpolation inequalities and some equivalent norms considering $\mathcal K_\lambda$, which make it possible to control the non-linearity involving the singularity $|x|^{-\tau}$ as well as the inverse square potential. The novelty here is the investigation of the energy critical regime which remains still open and the challenge is to deal with three technical problems: a non-local source term, an inhomogeneous singular term $|\cdot|^{-\tau}$, and the presence of an inverse square potential.

\end{abstract}
\maketitle
\tableofcontents
\vspace{ 1\baselineskip}
\renewcommand{\theequation}{\thesection.\arabic{equation}}
\section{Introduction}
In this paper we are concerned with the Cauchy problem for the inhomogeneous generalized Hartree equation with inverse square potential
\begin{equation}
\begin{cases}\label{S}
	i\partial_t u-\mathcal K_\lambda u=\epsilon |x|^{-\tau}|u|^{p-2}(I_\alpha *|\cdot|^{-\tau}|u|^p)u,\\
	u(x,0)=u_0(x), \quad (x,t)\in \mathbb{R}^n \times \mathbb{R},
\end{cases}
\end{equation}
where $p>2$, $\epsilon=\pm1$, and $\mathcal K_\lambda:=-\Delta+\frac\lambda{|x|^{2}}$ satisfying $\lambda>-\frac{(n-2)^2}{4}$.
Here the case $\epsilon =1$ is \textit{defocusing}, while the case $\epsilon =-1$ is \textit{focusing}. 
The Riesz potential is defined on $\mathbb{R}^n$ by 
$$I_\alpha:=\frac{\Gamma(\frac{n-\alpha}2)}{\Gamma(\frac\alpha2)\pi^\frac{n}22^\alpha}\,|\cdot|^{\alpha-n},\quad  0<\alpha<n.$$
The assumption on $\lambda$ comes from the sharp Hardy inequality \cite{abde},
\begin{equation}\label{prt}
	\frac{(n-2)^2}4\int_{\R^n}\frac{|f(x)|^2}{|x|^2}\,dx\leq \int_{\R^n}|\nabla f(x)|^2 dx,
\end{equation}
which guarantees that $\mathcal K_\lambda$ is thepositive self-adjoint extension of $-\Delta+\lambda/|x|^{-2}$. 
It is known that in the range $-\frac{(n-2)^2}4 <\lambda< 1-\frac{(n-2)^2}4$, the extension is not unique (see \cite{ksww,ect}). In such a case, one picks the Friedrichs extension (see \cite{ksww,pst}).

The problem \eqref{S} arises in various physical contexts. In the linear regime ($\epsilon=0$), the considered Schr\"odinger equation models quantum mechanics \cite{ksww,haa}. In the non-linear regime without potentials, namely $\lambda=0\neq\epsilon$, the equation \eqref{S} is of interest in the mean-field limit of large systems of non-relativistic bosonic atoms and molecules in a regime where the number of bosons is very large, but the interactions
between them are weak \cite{fl,hs,pg,mpt}. The homogeneous problem associated to the considered problem \eqref{S}, specifically, when $\lambda=\tau=0$, has several physical origins such as quantum mechanics \cite{pg,pgl} and Hartree-Fock theory \cite{ll}. The particular case $p=2$ and $\lambda=\tau=0$ is called standard Hartree equation. It is a classical limit of a field equation describing a quantum mechanical non-relativistic many-boson system interacting through a two body potential \cite{gvl}.

Now, let us return to the mathematical aspects of the generalized Hartree equation \eqref{S}. 
Recall the critical Sobolev index. 
If $u(x,t)$ is a solution of \eqref{S}, so is the family
$$u_\delta(x,t):=\delta^{\frac{2-2\tau+\alpha}{2(p-1)}} u(\delta x, \delta^2 t),$$ with the re-scaled initial data $u_{\delta,0}:=u_{\delta}(x,0)$ for all $\delta>0$.
Then, it follows that
\begin{equation*}
	\|u_{\delta,0}\|_{\dot H^1}=\delta^{1-\frac n2 +\frac{2-2\tau+\alpha}{2(p-1)}}\|u_0\|_{\dot H^1}.
\end{equation*}
If $p=1+\frac{2-2\tau+\alpha}{n-2}$, the scaling preserves the $\dot H^1$ norm of $u_0$, and in this case, \eqref{S} is referred as the energy-critical inhomogeneous generalized Hartree equation.
Moreover, the solution to \eqref{S} satisfies the mass and energy conservation,
where the mass conservation is
\begin{equation}
	\mathcal M[u(t)]:=\int_{\mathbb{R}^n} |u(x,t)|^2 dx = \mathcal M [u_0],
\end{equation}
and the energy conservation is
\begin{equation}
\mathcal E[u(t)]:= \int_{\mathbb{R}^n}\Big(|\nabla u|^2 +\lambda |x|^{-2} |u|^2\Big)\,dx+ \frac{\epsilon}{p}\mathcal P[u(t)]=\mathcal E[u_0],
\end{equation}
where the potential energy reads
$$\mathcal P[u(t)]:=\int_{\R^n} |x|^{-\tau}\big(I_\alpha *|\cdot|^{-\tau}|u|^p\big)|u|^p dx.$$

To the best of our knowledge, this paper is the first one dealing with the energy-critical inhomogeneous Hatree equation with inverse square potential, precisely \eqref{S} with $\lambda\neq0$. 
The main contribution is to develop a local well-posedness theory in the energy-critical case, as well as to investigate the blow-up of the solution in energy space for the inhomogeneous generalized Hartree equation \eqref{S}.
Precisely, the local theory is based on the standard contraction mapping argument via the availability of Strichartz estimates. More interestingly, we take advantage of some equivalent norms considering the operator $\mathcal{K}_\lambda$, namely $\|\sqrt{\mathcal{K}_\lambda}u\|_r\simeq\|u\|_{\dot W^{1,r}}$, which makes it possible to apply the contraction mapping principle without directly handling with the operator.
In the repulsive regime($\epsilon=-1$), we prove that the solution blows up in finite time without assuming the classical assumption such as radially symmetric or $|x|u_0 \in L^2$. 
The blow-up phenomenon is expressed in terms of the non-conserved potential energy, which may give a criteria in the spirit of \cite{vdd}, which implies in particular the classical phenomena under the ground state threshold in the spirit of \cite{km}. 

In this paper, we deal with three technical problems by the equation \eqref{S}, a non-local source term, the inhomogeneous singular term $|\cdot|^{-\tau}$, and the presence of an inverse square potential.
Indeed, in order to deal with the singular term $|\cdot|^{-\tau}$ in Lebesgue spaces, the method used in the literature decomposes the integrals on the unit ball and it's complementary (see, for example, \cite{mt}). 
However, this is no more sufficient to conclude in the energy critical case. For $\lambda=0$, the first author used some Lorentz spaces with the useful property $|\cdot|^{-\tau}\in L^{\frac{n}{\tau},\infty}$. 
To overcome these difficulties, we make use of some Caffarelli-Kohn-Nirenberg weighted interpolation inequalities which is different from the existing approaches.

Before stating our results, we introduce some Sobolev spaces defined in terms of the operator $\mathcal K_\lambda$ as the completion of $C^\infty_0(\R^n)$ with respect to the norms
\begin{align*}
\|u\|_{\dot W^{1,r}_\lambda}&:=\|\sqrt{\mathcal K_\lambda} u\|_{L^r} \quad \textnormal{and} \quad \|u\|_{W^{1,r}_\lambda}:=\|\langle \sqrt{\mathcal K_\lambda}\rangle u\|_{L^r},
\end{align*}
where $\langle \cdot\rangle:=(1+|\cdot|^2)^{1/2}$ and $L^r:=L^r(\R^n)$. We denote also the particular Hilbert cases $\dot W^{1,2}_\lambda=\dot H^1_\lambda$ and $W^{1,2}_\lambda=H^1_\lambda$.
We note that by the definition of the operator $\mathcal K_\lambda$ and Hardy estimate \eqref{prt}, one has
\begin{align*}
\|u\|_{\dot H^1_\lambda}&:=\|\sqrt{\mathcal K_\lambda}u\|=\big(\|\nabla u\|^2+\lambda\||x|^{-1}u\|^2\big)^\frac12\simeq \|u\|_{\dot H^1},
\end{align*}
where we write for simplicity $\|\cdot\|:=\|\cdot\|_{L^2(\R^n)}$.
\subsection{Well-posedness in the energy-critical case}
The theory of well-poseddness for the inhomogeneous Hartree equation ($\lambda=0$ in \eqref{S}) has been extensively studied in recent several years and is partially understood. (See, for examples, \cite{mt,sa, kls, sk} and references therein). For related results on the scattering theory, see also \cite{sx} for spherically symmetric datum and \cite{cx} in the general case.

Our first result is the following well-posedness in the energy-critical case.
\begin{thm}\label{loc}
	Let $n\ge3$, $\lambda >-\frac{(n-2)^2}{4}$ and $2\kappa=n-2-\sqrt{(n-2)^2+4\lambda}$. 
	Assume that 
	\begin{equation}\label{1.6}
		0<\alpha<n, \quad 2\kappa < n-2-\frac{2(n-2)}{3n-2+2\sqrt{9n^2+8n-16}}
	\end{equation}
	and
	\begin{equation}\label{1.7}
		\frac{\alpha}{2}-\frac{n+2+\sqrt{9n^2+8n-16}}{2}<\tau < \frac{\alpha}{2}-\max\{\frac{n-4}{2}, \frac{n-4}{n},\frac{\kappa}{n-2-2\kappa}-\frac{n}{4}\}.
	\end{equation}
	Then, for $u_0 \in H_{\lambda}^1(\mathbb{R}^n)$, there exist $T>0$ and a unique solution 
	$$u\in C([0,T]; H_\lambda^1) \cap L^q([0,T];W_{\lambda}^{1,r})$$
	to \eqref{S} with $p=1+\frac{2-2\tau+\alpha}{n-2}$
	 for any admissible pair $(q,r)$ in Definition \ref{dms}.
	Furthermore, the continuous dependence on the initial data holds.
\end{thm}

We also provide the small data global well-posedness and scattering results as follows:
\begin{thm}\label{glb}
Under the same conditions as in Theorem \ref{loc} and the smallness assumption on $\|u_0\|_{H_{\lambda}^1}$, there exists a unique global solution  
$$u\in C([0,\infty); H_\lambda^1) \cap L^q([0,\infty);W_{\lambda}^{1,r})$$
to \eqref{S} with $p=1+\frac{2-2\tau+\alpha}{n-2}$ for any admissible pair $(q,r)$.
	Furthermore, the solution scatters in $H_\lambda^1$, i.e., there exists $\phi\in H_\lambda^1$ such that
	$$\lim_{t\to\infty} \|u(t)-e^{-it\mathcal{K}_{\lambda}}\phi\|_{H_{\lambda}^1}=0.$$
\end{thm}

\subsection{Blow-up of energy solutions}
We now turn to our attention to blow-up of the solution to \eqref{S} under the ground state threshold, in the focusing regime. A particular global solution of \eqref{S} with $\epsilon=-1$ is the stationary solution to \eqref{S}, namely
\begin{equation}\label{E}
-\Delta \varphi+\frac{\lambda}{|x|^2}\varphi=|x|^{-\tau}|\varphi|^{p-2}(I_\alpha *|\cdot|^{-\tau}|\varphi|^p)\varphi,\quad 0\neq \varphi\in {H^1_\lambda}.
\end{equation}
Such a solution called ground state plays an essential role in the focusing regime. 
The following result is the existence of ground states to \eqref{E}.
\begin{thm}\label{gag}
{Let $n\geq3$, $\lambda>-\frac{(n-2)^2}4$ and $p=1+\frac{2-2\tau+\alpha}{n-2}$. Assume that 
\begin{equation}\label{as}
0<\alpha<n \quad \text{and} \quad 0<\tau< 1+\frac{\alpha}{n}.
\end{equation}\label{inte}
Then, the following inequality holds:
\begin{equation}\label{gagg}
	\int_{\mathbb{R}^n} |x|^{-\tau}|u|^p (I_\alpha \ast |\cdot|^{-\tau}|u|^p) \leq C_{N,\tau,\alpha,\lambda} \big\|\sqrt{\mathcal{K}_{\lambda} }u\big\|^{2p}.
	\end{equation}}
Moreover, there exists $\varphi\in H_{\lambda}^1$ a ground state solution to \eqref{E}, which is a minimizing of the problem
\begin{equation}\label{min}
    \frac1{C_{N,\tau,\alpha,\lambda}}=\inf\Big\{\frac{\|\sqrt{\mathcal K_\lambda} u\|^{2p}}{\mathcal P[u]},\quad0\neq u\in H^1_\lambda\Big\}.
    \end{equation}
\end{thm}

\begin{rems}
\begin{enumerate}
\item[1.] Theorem \ref{gag} does not require to assume that $p\geq2$;
\item[2.] $C_{N,\tau,\alpha,\lambda}$ denotes the best constant in the inequality \eqref{gagg};
\item[3.] compared with the homogeneous regime $\tau=0$, the minimizing \eqref{min} is never reached for $\lambda>0$, see \cite{kmvzz}.
\end{enumerate}
\end{rems}

Here and hereafter, we denote $\varphi$ a ground state solution of \eqref{E} and the scale invariant quantities
\begin{align*}
\mathcal{ME}[u_0]:=\frac{\mathcal E[u_0]}{\mathcal E[\varphi]},\quad
\mathcal{MG}[u_0]:=\frac{\|\sqrt{\mathcal K_\lambda} u_0\|}{\|\sqrt{\mathcal K_\lambda} \varphi\|},\quad
\mathcal{MP}[u_0]:=\frac{\mathcal P[u_0]}{\mathcal P[\varphi]}.
\end{align*}
The next theorem gives a blow-up phenomenon in the energy-critical focusing regime under the ground state threshold.
\begin{thm}\label{t1}
Under the assumptions in \ref{loc} and $\epsilon=-1$, let $\varphi$ be a ground state solution to \eqref{E} and $u\in C_{T^*}(H^1_\lambda)$ be a maximal solution of the focusing problem \eqref{S}.
If 
\begin{equation} \label{ss'}
	\sup_{t\in[0,T^*)}\mathcal I[u(t)]<0,
\end{equation}
then $u$ blows-up in finite or infinite time.
Here, $\mathcal I[u]:=\|\sqrt{\mathcal K_\lambda} u\|^2-\mathcal{P}[u]$.
\end{thm}
\begin{rems}
	\begin{enumerate}
		\item[1.]
$u$ blows-up in infinite time means that it is global and there is $t_n\to\infty$ such that $\|\sqrt{\mathcal K_\lambda} u(t_n)\|\to\infty$;	
  	  \item[2.]
 the threshold is expressed in terms of the potential energy $\mathcal P[u]$, which is a non conserved quantity;
  	 \item[3.]
  		the theorem here doesn't require the classical assumptions such as spherically symmetric data or $|x|u_0\in L^2$;
		\item[4.]
		a direct consequence of the variance identity is that the energy solution to \eqref{S} blows-up in finite time if $|x|u_0\in L^2$ and \eqref{ss'} is satisfied;
	\end{enumerate}
\end{rems}

The next result is a consequence of Theorem \ref{t1}.
\begin{cor}\label{t2}
Under the assumptions in Theorem \ref{loc} and $\epsilon=-1$, Let $\varphi$ be a ground state of \eqref{E} and $u_0\in H^1_\lambda$ such that
\begin{equation} \label{t11}
\mathcal{ME}[u_0]<1.
\end{equation}
If we assume that
\begin{equation}\label{t13}
	\mathcal{MG}[u_0]>1,\end{equation}
then the energy solution of \eqref{S} blows-up in finite or infinite time 
\end{cor}
\begin{rems}
\begin{enumerate}
\item[1.]
The assumptions of the above result are more simple to check than \eqref{ss'}, because they are expressed in terms of conserved quantities;
\item[2.]
the above ground state threshold has a deep influence in the NLS context since the pioneering papers \cite{km,Holmer};
\end{enumerate} 
\end{rems}

Finally, we close this subsection with some additional results which gives the boundedness of the energy solution.
\begin{prop}\label{s}
Under the assumptions in Theorem \ref{loc} and $\epsilon=-1$, let $\varphi$ be a ground state solution to \eqref{E} and $u\in C_{T^*}(H^1_\lambda)$ be a maximal solution of the focusing problem \eqref{S}. 
If \begin{equation} \label{ss}
	\sup_{t\in[0,T^*)}\mathcal{MP}[u(t)]<1,
\end{equation}
then 
$u$ is bounded in $H^1_\lambda$.  
\end{prop}
The next is a consequence of Proposition \ref{s}
\begin{cor}\label{s2}
Under the assumptions in Theorem \ref{loc} and $\epsilon=-1$. Let $\varphi$ be a ground state of \eqref{E} and $u_0\in H^1_\lambda$ satisfying \eqref{t11}
If
\begin{equation}\label{t12}
	\mathcal{MG}[u_0]<1, \end{equation}
then the energy solution of \eqref{S} is bounded.
\end{cor}
\begin{rem}
\begin{enumerate}
	\item[1.]	
the global existence and energy scattering under the assumptions \eqref{ss} in Proposition \ref{s} or \eqref{t11}-\eqref{t12} in Corollary \ref{s2} is investigated in a paper in progress.
\end{enumerate}
\end{rem}

The rest of this paper is organized as follows. In Section 2 we introduce some useful properties that we need. Section 3 develops a local theory and a global one for small datum. In section 4, the existence of ground states is established. Section 5 establishes blow-up of solutions under the ground state threshold and the boundedness of energy solutions. In the appendix, a Morawetz type estimate is proved.

Throughout this paper, the letter $C$ stands for a positive constant which may be different at each occurrence. We also denote $A \lesssim B$ to mean $A \leq CB$ with unspecified constants $C>0$.

\section{Preliminaries}
In this section, we introduce some useful properties which will be utilized throughout this paper. 
We also recall the Strichartz estimates.
Let us start with the Hardy-Littlewood-Sobolev inequality \cite{el} which is suitable for dealing with non-local source term in \eqref{S}:
\begin{lem}\label{hls}
Let $n\geq1$ and $0 <\alpha < n$.
\begin{enumerate}
\item[1.]
Let $s\geq1$ and $r>1$ such that $\frac1r=\frac1s+\frac\alpha n$. Then,
$$\|I_\alpha*g\|_{L^s}\leq C_{n,s,\alpha}\|g\|_{L^r}.$$
\item[2.]
Let $1<s,r,t<\infty$ be such that $\frac1r +\frac1s=\frac1t +\frac\alpha n$. Then,
$$\|f(I_\alpha*g)\|_{L^t}\leq C_{n,t,\alpha}\|f\|_{L^r}\|g\|_{L^s}.$$
\end{enumerate}
\end{lem}

The following lemma is a weighted version of the Sobolev embedding, that is, a special case of Caffarelli-Kohn-Nirenberg weighted interpolation inequalities \cite{sgw,csl}:
{\begin{lem}\label{ckn}
		Let $n\geq1$ and 
		$$1< p\leq q<\infty, \quad -\frac nq<b\leq a<\frac n{p'} \quad \text{and} \quad a-b-1=n\Big(\frac1q-\frac1p\Big).$$ 
		Then,
		$$\||x|^{b}f\|_{L^q}\leq C\||x|^a\nabla  f\|_{L^p}.$$
\end{lem}}

Now, we describe several properties related to the operator $\mathcal K_\lambda.$
Since $\|f\|_{H^1} \simeq \|f\|_{H_\lambda^1}$, one has the following compact Sobolev injection (\cite[Lemma 3.1]{cg}):
\begin{lem}\label{compact}
Let $n\geq3$, $0<\tau<2$ and $2<r<\frac{2(n-\tau)}{n-2}$. Then, the following injection is compact:
	$$H^1_\lambda\hookrightarrow\hookrightarrow L^{r}(|x|^{-\tau}\,dx) .$$
\end{lem}

We also have the following equivalent norms to Sobolev ones, see \cite{kmvzz} and \cite[Remark 2.2]{cg}:
\begin{lem}\label{2.2}
Let $n\geq3$, $\lambda>-\frac{(n-2)^2}4$, $1<r<\infty$ and $2\kappa=n-2-\sqrt{(n-2)^2+4\lambda}$. Then,
	\begin{enumerate}
		\item[1.]
		if $\frac{1+\kappa}n<\frac1r<\min\{1,1-\frac\kappa n\}$, thus, $\|f\|_{\dot W^{1,r}}\lesssim \|f\|_{\dot W_\lambda^{1,r}}$
		\item[2.]
		if $\max\{\frac{1}n,\frac{\kappa}n\}<\frac1r<\min\{1,1-\frac\kappa N\}$, thus, $\|f\|_{\dot W_\lambda^{1,r}}\lesssim\|f\|_{\dot W^{1,r}}$
	\end{enumerate}
\end{lem}

Finally, we recall the Strichartz estimates. As we shall see, the availability of these estimates is the key role in the proof of Theorem \ref{loc}.
\begin{defi}\label{dms}
Let $n\ge3$. We say that $(q,r)$ is an admissible pair if it satisfies 
\begin{equation*} 	
	2\le q \le \infty, \quad 2\le r \le \frac{2n}{n-2} \quad \text{and} \quad \frac2q+\frac{n}{r}=\frac{n}{2}.
	\end{equation*}
\end{defi}
\begin{prop}\cite{bpst,zz,df}\label{str}
Let $n\geq3$, $\lambda>-\frac{(n-2)^2}4$. Then, there exists $C>0$ such that 
\begin{enumerate}
\item[1.]
$\|e^{-it\mathcal K_\lambda}f\|_{L_t^q(L_x^r)}\leq C\|f\|,$
\item[2.]
$\|\int_0^{t_1}e^{-i(t-t_1)\mathcal K_\lambda}F(\cdot,t_1)dt_1\|_{L_t^q(L_x^r)}\leq C\|F\|_{L_t^{\tilde q'}L_x^{\tilde r'}}.$
\end{enumerate}
\end{prop}

Finally, one gives a classical Morawetz estimate proved in the appendix. Let $\phi:\R^n\to\R$ be a smooth function and define the variance potential
$$V_\phi(t):=\int_{\R^n}\phi(x)|u(x,t)|^2\,dx,$$
and the Morawetz action
$$M_\phi(t):=2\Im\int_{\R^n} \bar u(\xi_ju_j)\,dx=2\Im\int_{\R^n} \bar u(\nabla\phi\cdot\nabla u)\,dx.$$
\begin{prop}\label{mrwz}
	Let $\phi:\mathbb{R}^N \rightarrow \mathbb{R}$ be a radial, real-valued multiplier, $\phi=\phi(|x|)$.
	Then, for any solution $u\in C([0,T];H_\lambda^1)$ of the generalized Hartree equation \eqref{S} in the focusing sign with initial data $u_0\in H_{\lambda}^{1}$, the following virial-type identities hold:
	\begin{equation*}
		V'_\phi(t)=
		2\Im\int_{\R^n} \bar u\nabla\phi\cdot\nabla u dx
	\end{equation*}
	and
	\begin{align*}
		V''_\phi(t)=M_\phi'(t)&=4\sum_{k,l=1}^{N}\int_{\R^n}\partial_l\partial_k\phi\Re(\partial_ku\partial_l\bar u)dx-\int_{\R^n}\Delta^2\phi|u|^2dx+4\lambda\int_{\R^n}\nabla\phi\cdot x\frac{|u|^2}{|x|^4}dx\\
		&\qquad-\frac{2(p-2)}{p}\int_{\R^n}\Delta\phi|x|^{-\tau}|u|^p(I_\alpha*|\cdot|^{-\tau}|u|^{p})dx\\
		&\qquad\qquad -\frac{4\tau}p\int_{\R^n}x\cdot\nabla\phi|x|^{-\tau-2}|u|^{p}(I_\alpha*|\cdot|^{-\tau}|u|^p)dx\\
		&\qquad\qquad -\frac{4(N-\alpha)}p\sum_{k=1}^N\int_{\R^n}|x|^{-\tau}|u|^{p}\partial_k\phi(\frac{x_k}{|\cdot|^2}I_\alpha*|\cdot|^{-\tau}|u|^p)dx.
	\end{align*}
\end{prop}

{\section{Well-posedness in the energy space}}
In this section, we develop a local theory in the energy space $H_{\lambda}^1$, Theorem \ref{loc}. Moreover, we prove Theorem \ref{glb} about the global theory for small datum. Let us first denote the source term 
\begin{align*}
\mathcal N&:=\mathcal N[u]:=|x|^{-\tau}|u|^{p-2}(I_\alpha *|\cdot|^{-\tau}|u|^p)u.
\end{align*}
\subsection{Nonlinear estimates} We first establish some nonlinear estimates for $\mathcal{N}[u]$.  These nonlinear estimates will play an important role in proving the well-posedness results applying the contraction mapping principle. Before stating the nonlinear estimates, we introduce some notations. 
We set 
$$\mathcal A=\{(q,r):(q,r) \,\,\text{is}\,\, \text{admissible}\},$$
and then define the norm
$$\|u\|_{\Lambda(I)}=\sup_{(q,r)\in \mathcal A}\|u\|_{L_t^q(I;L_x^r)}$$
and its dual weighted norm 
$$\|u\|_{\Lambda'(I)}=\sup_{(\tilde q,\tilde r)\in \mathcal A}\|u\|_{L_t^{\tilde q'}(I;L_x^{\tilde r'})}$$
for any interval $I\subset \mathbb{R}.$

\begin{lem}\label{non}
Let $n\ge3$, $\lambda >-\frac{(n-2)^2}{4}$ and $p=1+\frac{2-2\tau+\alpha}{n-2}$. 
Assume that  
\begin{equation}\label{ass1}
0<\alpha<n, \quad 2\kappa < \frac{5n-4-\sqrt{9n^2+8n-16}}{2}
\end{equation}
and
\begin{equation}\label{ass2}
\frac{\alpha}{2}-\frac{n+2+\sqrt{9n^2+8n-16}}{2}<\tau < \frac{\alpha}{2}-\max\{\frac{n-4}{2}, \frac{n-4}{n},\frac{\kappa}{n-2-2\kappa}-\frac{n}{4}\}.
\end{equation}
Then, {there is an admissible pair $(q,r)$, such that}
\begin{equation}\label{non1}
\|\sqrt{\mathcal{K}_{\lambda}}\,\mathcal{N}[u]\|_{\Lambda'(0,T)} \leq C \|u\|_{L_t^{q}(I;\dot W_{\lambda}^{1,r})}^{2p-1}
\end{equation}
and
\begin{equation}\label{non2}
\|\mathcal{N}[u]-\mathcal{N}[v]\|_{\Lambda'(0,T)} \leq C (\|u\|_{L_t^q(I;\dot W_{\lambda}^{1,r})}^{2p-2} +\|u\|_{L_t^q(I;\dot W_{\lambda}^{1,r})}^{2p-2})\|u-v\|_{L_t^q(I;L_x^r)}.
\end{equation}
\end{lem}

\begin{proof}
It is sufficient to show that there exist $(q,r)$ for which 
\begin{equation}\label{non11}
\|\nabla\mathcal{N}[u]\|_{L_t^2(I;L_x^{\frac{2n}{n+2}})} \leq C \|\nabla u\|_{L_t^{q}(I;L_x^r)}^{2p-1}
\end{equation}
and
\begin{equation}\label{non22}
\|\mathcal{N}[u]-\mathcal{N}[v]\|_{L_t^2(I;L_x^{\frac{2n}{n+2}})} \leq C \|\nabla u\|_{L_t^q(I;L_x^r)}^{2p-2} \|u-v\|_{L_t^q(I;L_x^r)}
\end{equation}
hold for $\lambda, \alpha, \kappa, \tau, p$ given as in the lemma.
Indeed, by applying the equivalent norms to Sobolev ones (see Lemma \ref{2.2}) we obtain the desired estimates \eqref{non1} and \eqref{non2} if 
\begin{equation}\label{ka}
\max\{\frac1n,\frac{\kappa}n\}<\frac{n+2}{2n}<\min\{1,\frac{n-\kappa}n\}, \quad \frac{1+\kappa}{n}<\frac1{r}<\min\{1,\frac{n-\kappa}{n}\},
\end{equation}
Here, one can easily see that the first condition in \eqref{ka} is always satisfied.

Now we start to prove \eqref{non11}.
Let us set
\begin{equation}\label{st}
\frac1q=\frac{1}{2(2p-1)}, \quad \frac{n-2}{2n}\leq \frac1r \leq \frac12, \quad \frac{2}q+\frac{n}{r}=\frac{n}{2}, \quad 0<\alpha<n.
\end{equation}
We first see that 
\begin{align*}
\|&\nabla\mathcal{N}[u]\|_{L_x^{\frac{2n}{n+2}}}\\
&\lesssim\big\||x|^{-\tau-1}|u|^{p-1}(I_\alpha\ast|\cdot|^{-\tau}|u|^p)\big\|_{L_x^{\frac{2n}{n+2}}}+\big\||x|^{-\tau}|u|^{p-1}(I_\alpha \ast|\cdot|^{-\tau-1}|u|^{p})\big\|_{L_x^{\frac{2n}{n+2}}}\\
&+\||x|^{-\tau}|u|^{p-2}|\nabla u|(I_\alpha\ast|\cdot|^{-\tau}|u|^p)\|_{L_x^{\frac{2n}{n+2}}}+\||x|^{-\tau}|u|^{p-1}(I_\alpha\ast|\cdot|^{-\tau}|u|^{p-1}\nabla u)\|_{L_x^{\frac{2n}{n+2}}}\\
&:=A_1+A_2+A_3+A_4.
\end{align*}
The first term $A_1$ is bounded by using Lemmas \ref{hls} and \ref{ckn},
in turn,
\begin{align*}
\big\||x|^{-\tau-1}|u|^{p-1}|I_\alpha\ast|\cdot|^{-\tau}|u|^p|\big\|_{L_x^{\frac{2n}{n+2}}} 
&\lesssim \||x|^{-\tau-1}|u|^{p-1}\|_{L_x^{a_1}} \||x|^{-\tau}|u|^p\|_{L_x^{b_1}} \\
&=\||x|^{-\frac{\tau+1}{p-1}}u\|^{p-1}_{L_x^{(p-1)a_1}}\||x|^{-\frac{\tau}{p}}u\|^{p}_{L_x^{pb_1}}\\
&\lesssim \|\nabla u\|_{L_x^r}^{2p-1},
\end{align*}
if 
\begin{equation}\label{c7}
0<\frac1{a_1},\frac1{b_1}<1, \quad  \frac1{a_1}+\frac1{b_1}=\frac{n+2}{2n}+\frac{\alpha}{n},
\end{equation}
\begin{equation}\label{c8}
0<\frac{1}{(p-1)a_1}\leq \frac1r\leq1,\quad  0\leq\frac{\tau+1}{p-1}<\frac{n}{(p-1)a_1} ,\quad  \frac{\tau+1}{p-1}-1=\frac{n}{(p-1)a_1}-\frac{n}{r}
\end{equation}
\begin{equation}\label{c9}
0<\frac{1}{pb_1}\leq \frac1r\leq 1,\quad 0 \leq \frac{\tau}{p}<\frac{n}{pb_1},\quad \frac{\tau}{p}-1=\frac{n}{pb_1}-\frac{n}{r}.
\end{equation}

Similarly, 
\begin{align*}
A_2&=\big\||x|^{-\tau}|u|^{p-1}(I_\alpha \ast|\cdot|^{-\tau-1}|u|^{p})\big\|_{L_x^{\frac{2n}{n+2}}}\\
&\leq \||x|^{-\tau}|u|^{p-1}\|_{L_x^{a_2}}\||x|^{-\tau-1}|u|^p\|_{L_x^{b_2}}\\
&\lesssim \|\nabla u\|^{2p-1}_{L_x^r}
\end{align*}
if 
\begin{equation}\label{c10}
0<\frac1{a_2}, \frac{1}{b_2}<1, \quad \frac{1}{a_2}+\frac{1}{b_2}=\frac{n+2}{2n}+\frac{\alpha}{n},
\end{equation}
\begin{equation}\label{c11}
0<\frac1{(p-1)a_2}\leq \frac1r\leq 1, \quad 0\leq \frac{\tau}{p-1}<\frac{n}{(p-1)a_2}, \quad \frac{\tau}{p-1}-1=\frac{n}{(p-1)a_2}-\frac{n}{r},
\end{equation}
\begin{equation}\label{c12}
0<\frac{1}{pb_2}\leq \frac1r \leq 1, \quad 0\leq \frac{\tau+1}{p}<\frac{n}{pb_2}, \quad \frac{\tau+1}{p}-1=\frac{n}{pb_2}-\frac{n}{r}.
\end{equation}

The third term $A_3$ is bounded by using Lemma \ref{hls}, H\"older's inequality and Lemma \ref{ckn} 
in turn as
\begin{align*}
\big\||x|^{-\tau}|u|^{p-2}|\nabla u||I_\alpha \ast|\cdot|^{-\tau}|u|^{p}|\big\|_{L_x^{\frac{2n}{n+2}}}
&\lesssim \||x|^{-\tau}|u|^{p-2}|\nabla u|\|_{L_x^{a_1}} \||x|^{-\tau}|u|^p\|_{L_x^{b_1}}\\
&\leq \||x|^{-\tau}|u|^{p-2}\|_{L_x^{a_3}}\|\nabla u\|_{L_x^r} \||x|^{-\tau}|u|^p\|_{L_x^{b_1}}\\
&=\||x|^{-\frac{\tau}{p-2}}u\|^{p-2}_{L_x^{(p-2)a_3}} \|\nabla u\|_{L_x^r} \||x|^{-\frac{\tau}{p}}u\|^p_{L_x^{pb_1}}\\
&\lesssim \|\nabla u\|_{L_x^r}^{2p-1}
\end{align*}
if 
\begin{equation}\label{c1}
0<\frac1{a_1}, \frac1{b_1}<1, \quad \frac{1}{a_1}+\frac{1}{b_1}=\frac{n+2}{2n}+\frac{\alpha}{n},\quad \frac{1}{a_1}=\frac{1}{a_3}+\frac{1}{r}, 
\end{equation}
\begin{equation}\label{c2}
0<\frac{1}{(p-2)a_3} \leq \frac1r\leq 1, \quad 0\leq \frac{\tau}{p-2}<\frac{n}{(p-2)a_3}, \quad \frac{\tau}{p-2}-1=\frac{n}{(p-2)a_3}-\frac{n}{r},
\end{equation}
\begin{equation}\label{c3}
0<\frac1{pb_1}\leq \frac1r\leq 1, \quad 0\leq \frac{\tau}{p}<\frac{n}{pb_1}, \quad \frac{\tau}{p}-1=\frac{n}{pb_1}-\frac{n}{r}.
\end{equation}

Similarly,
\begin{align*}
A_4&=\||x|^{-\tau}|u|^{p-1}(I_\alpha\ast|\cdot|^{-\tau}|u|^{p-1}\nabla u)\|_{L_x^{\frac{2n}{n+2}}}\\
&\lesssim \||x|^{-\tau}|u|^{p-1}\|_{L_x^{a_2}}\||x|^{-\tau}|u|^{p-1}\nabla u\|_{L_x^{b_2}}\\
&\leq \||x|^{-\frac{\tau}{p-1}}|u|\|^{p-1}_{L_x^{(p-1)a_2}}\||x|^{-\frac{\tau}{p-1}}u\|^{p-1}_{L_x^{(p-1)b_4}}\|\nabla u\|_{L_x^r}\\
&\lesssim \|\nabla u\|_{L_x^r}^{2p-1}
\end{align*}
if 
\begin{equation}\label{c4}
0<\frac1{a_2}, \frac{1}{b_2}<1,\quad \frac1{a_2}+\frac1{b_2}=\frac{n+2}{2n}+\frac{\alpha}{n},\quad \frac{1}{b_2}=\frac{1}{b_4}+\frac1r
\end{equation}
\begin{equation}\label{c5}
0<\frac{1}{(p-1)a_2}\leq \frac1r\leq 1, \quad 0\leq \frac{\tau}{p-1}\leq \frac{n}{(p-1)a_2}, \quad \frac{\tau}{p-1}-1=\frac{n}{(p-1)a_2}-\frac{n}{r}, 
\end{equation}
\begin{equation}\label{c6}
0<\frac1{(p-1)b_4}\leq \frac1r\leq 1, \quad 0\leq \frac{\tau}{p-1}<\frac{n}{(p-1)b_4}, \quad \frac{\tau}{p-1}-1=\frac{n}{(p-1)b_4}-\frac{n}{r}.
\end{equation}

On the other hand, in order to show \eqref{non22}, we first use the following simple inequality
\begin{align*}
\big|\mathcal N[u]-\mathcal N[v]\big| &\lesssim \Big||x|^{-\tau}(|u|^{p-2}+|v|^{p-2})|u-v|(I_\alpha\ast|\cdot|^{-\tau}|u|^p)\Big| 
\\&\qquad\quad+\Big||x|^{-\tau}|v|^{p-1}\big(I_\alpha \ast|\cdot|^{-\tau}(|u|^{p-1}+|v|^{p-1})|u-v|\big)\Big|.
\end{align*}
From this, we see that 
\begin{align*}
\|\mathcal N[u]-\mathcal N[v]\|_{L_x^{\frac{2n}{n+2}}} &\leq \||x|^{-\tau}(|u|^{p-2}+|v|^{p-2})|u-v|(I_\alpha\ast|\cdot|^{-\tau}|u|^p)\|_{L_x^{\frac{2n}{n+2}}}\\
&\qquad+\||x|^{-\tau}|v|^{p-1}\big(I_\alpha \ast|\cdot|^{-\tau}(|u|^{p-1}+|v|^{p-1})|u-v|\big)\|_{L_x^{\frac{2n}{n+2}}}\\
&:=B_1+B_2.
\end{align*}

Replacing $\nabla u$ with $u-v$ in the process of dealing with $A_3$, we get 
\begin{align*}
B_1&=\||x|^{-\tau}(|u|^{p-2}+|v|^{p-2})|u-v|(I_\alpha\ast|\cdot|^{-\tau}|u|^p)\|_{L_x^{\frac{2n}{n+2}}}\\
&\lesssim \Big(\||x|^{-\tau}|u|^{p-2}|u-v|\|_{L_x^{a_1}}+\||x|^{-\tau}|v|^{p-2}|u-v|\|_{L_x^{a_1}}\Big)\||x|^{-\tau}|u|^p\|_{L_x^{b_1}}\\
&\lesssim \Big(\||x|^{-\frac{\tau}{p-2}}u\|_{L_x^{(p-2)a_3}}^{p-2}+ \||x|^{-\frac{\tau}{p-2}}u\|_{L_x^{(p-2)a_3}}^{p-2}\Big)\|u-v\|_{L_x^r}\||x|^{-\frac{\tau}{p}}u\|_{L_x^{pb_1}}^p\\
&\lesssim \big(\|\nabla u\|_{L_x^{r}}^{2p-2}+\|\nabla v\|_{L_x^{r}}^{2p-2}\big)\|u-v\|_{L_x^r}
\end{align*}
under the conditions \eqref{c1}, \eqref{c2} and \eqref{c3}. 
Similarly, replacing $\nabla u$ with $u-v$ in estimating $A_4$, we also have
\begin{align*}
B_2&=\||x|^{-\tau}|u|^{p-1}(I_\alpha\ast|\cdot|^{-\tau}(|u|^{p-1}+|v|^{p-1})|u-v|)\|_{L_x^{\frac{2n}{n+2}}}\\
&\lesssim \||x|^{-\tau}|u|^{p-1}\|_{L_x^{a_2}}\||x|^{-\tau}(|u|^{p-1}+|v|^{p-1})|u-v|\|_{L_x^{b_2}}\\
&\lesssim \||x|^{-\frac{\tau}{p-1}}u\|_{L_x^{(p-1)a_2}}^{p-1} \Big(\||x|^{-\frac{\tau}{p-1}}u\|_{L_x^{(p-1)b_4}}^p+\||x|^{-\frac{\tau}{p-1}}v\|_{L_x^{(p-1)b_4}}^p\Big)\|u-v\|_{L_x^r}\\
&\lesssim \big(\|\nabla u\|_{L_x^{r}}^{2p-2}+\|\nabla v\|_{L_x^{r}}^{2p-2}\big)\|u-v\|_{L_x^r}
\end{align*}
under the conditions \eqref{c4}, \eqref{c5} and \eqref{c6}.

Now it remains to eliminate some redundant pairs, we then show that there exists an admissible pair $(q,r)$ satisfying the assumptions in the lemma.
The third conditions of \eqref{c8} and \eqref{c9} can be rewritten with respect to $a$ and $b$, respectively, as follow:
\begin{equation}\label{c19}
\frac{n}{a_1}=\frac{(p-1)n}{r}+\tau-p+2, \quad \frac{n}{b_1}=\frac{pn}{r}+\tau-p.
\end{equation}
Inserting these into the second condition of \eqref{c7} implies 
\begin{equation}\label{a}
\frac{(2p-1)n}{r}=\alpha+2p-2\tau-1+\frac{n}{2}.
\end{equation}
Here, we note that this equation is equivalent to the second condition of \eqref{c10} by using \eqref{c18}.
Inserting \eqref{c19} into the conditions in \eqref{c7}, \eqref{c8} and \eqref{c9}, these conditions are summarized as follows:
\begin{equation}\label{c15}
\frac{p-\tau-2}{p-1}<\frac{n}{r}<\frac{p-\tau-2+n}{p-1}, \quad \frac{p-\tau}{p}<\frac{n}{r}<\frac{p-\tau+n}{p},
\end{equation}
\begin{equation}\label{c13}
\tau-p+2\leq 0, \quad \frac{n}{r}\leq n, \quad \frac{p-\tau-2}{p-1} \leq 1 <\frac{n}{r}
\end{equation}
\begin{equation}\label{c14}
\tau-p\leq 0 , \quad \frac{p-\tau}{p}\leq 1 < \frac{n}{r}
\end{equation}
Since $\tau>0$, the first inequalities of the last conditions in \eqref{c13} and \eqref{c14} are redundant.
The first condition in \eqref{c14} is also redundant by the first one in \eqref{c13}.

Also, the third conditions of \eqref{c11} and \eqref{c12} can be rewritten as
\begin{equation}\label{c18}
\frac{n}{a_2}=\frac{(p-1)n}{r}+\tau-p+1, \quad \frac{n}{b_2}=\frac{pn}{r}+\tau-p+1
\end{equation}
Inserting these into the conditions \eqref{c10}, \eqref{c11} and \eqref{c12}, these conditions are summarized as
\begin{equation}\label{c16}
\frac{p-\tau-1}{p-1}<\frac{n}{r}<\frac{p-\tau-1+n}{p-1}, \quad \frac{p-\tau-1}{p}<\frac{n}{r}<\frac{p-\tau-1+n}{p}
\end{equation}
\begin{equation}\label{c17}
\tau-p+1\leq 0, \quad \frac{n}{r}\leq n, \quad \frac{p-\tau-1}{p-1}\leq 1 <\frac{n}{r}, \quad \frac{p-\tau-1}{p}\leq 1 < \frac{n}{  r}.
\end{equation}
Here, since $\tau>0$, the first inequalities of the last two conditions in \eqref{c17} are redundant.
The first conditions in \eqref{c17} is also eliminated by the first one of \eqref{c13}.

Finally, the first two conditions of \eqref{c2} and \eqref{c6} are summarized by inserting the third conditions of \eqref{c2} and \eqref{c6} as
\begin{equation}\label{c20}
\frac{p-\tau-2}{p-2}<\frac{n}r\leq n, \quad \tau-p+2\leq0, \quad \frac{p-\tau-2}{p-2}\leq 1<\frac{n}{r},
\end{equation}
\begin{equation}\label{c21}
\frac{p-\tau-1}{p-1}<\frac{n}{r}\leq n, \quad \tau-p+1\leq 0, \quad \frac{p-\tau-1}{p-1}\leq 1 <\frac{n}{r}.
\end{equation}
Here, the second condition in \eqref{c21} is eliminated by the second one in \eqref{c20}.
Since $p>2$ and $\tau>0$, all lower bounds of $n/r$ in \eqref{c15}, \eqref{c16}, \eqref{c20} and \eqref{c21} are eliminated by $1$. 
Moreover, by using $p>2$ the upper bounds of $n/r$ in the second condition of \eqref{c15} and the first one of \eqref{c16} are also eliminated by the upper one of $n/r$ in the second condition of \eqref{c16}.
As a result, combining all the above conditions, we get 
\begin{equation}\label{rr}
1<\frac{n}{r}<\min\Big\{\frac{p-\tau-2+n}{p-1}, \frac{p-\tau-1+n}{p}\Big\}, \quad 0<\tau\leq p-2.
\end{equation}

On the other hand, substituting the first condition into the third one in \eqref{st} implies
\begin{equation}\label{r}
\frac{n}{r}=\frac{n}{2}-\frac1{2p-1}.
\end{equation}
Note that \eqref{a} is exactly same as $p=1+\frac{2-2\tau+\alpha}{n-2}$ by substituting \eqref{r} into \eqref{a}. 
Eliminating $r$ by inserting \eqref{r} into the second conditios of \eqref{ka} and \eqref{st}, the first one of \eqref{rr}, we then get
\begin{equation}\label{c25}
1+\kappa < \frac{n}{2}-\frac1{2p-1}<\min\{n,n-\kappa\}, \quad \frac{n-2}{2}\leq \frac{n}{2}-\frac1{2p-1}\leq \frac{n}2,
\end{equation}
\begin{equation}\label{c22}
p>\frac{n}{2(n-2)}, \quad \tau <n-1+\min\Big\{(p-1)\Big(\frac{2p}{2p-1}-\frac{n}{2}\Big), p\Big(\frac{2p}{2p-1}-\frac{n}{2}\Big)\Big\}.
\end{equation}
Here, the first condition in \eqref{c25} can be divided into two inequalities
\begin{equation}\label{c26}
p>\frac12+\frac{1}{n-2-2\kappa}, \quad \max\big\{-\frac{n}{2},-\frac{n}{2}+\kappa\big\}<\frac1{2p-1},
\end{equation}
in which the second condition is redundant since the maximum value is always negative.
Since $p>2$ and $n\ge3$, the last condition in \eqref{c25} and the first condition in \eqref{c22} are redundant.
Moreover, since $\frac{2p}{2p-1}-\frac{n}{2}<0$, the last condition in \eqref{c22} is reduced
\begin{equation}\label{c23}
\tau <n-1+ p\Big(\frac{2p}{2p-1}-\frac{n}{2}\Big).
\end{equation}

In order to eliminate $\alpha$, inserting $p=1+\frac{2-2\tau+\alpha}{n-2}$ into the last condition in \eqref{st}, we also have
\begin{equation}\label{c24}
\frac{n}{2}-\frac{(n-2)p}{2}<\tau<n-\frac{(n-2)p}{2}.
\end{equation}
Now we make the lower bounds of $\tau$ less than the upper ones of $\tau$ in \eqref{rr}, \eqref{c23} and \eqref{c24} to obtain 
\begin{align*}
\nonumber
\max\Big\{2, &\frac{5n-4-\sqrt{9n^2+8n-16}}{4(n-2)},\frac{n+4}{n}, \frac{n-2}{2(n-1)}\Big\} \leq \\
&\qquad\qquad\qquad\qquad\quad p \leq \min\Big\{\frac{2n}{n-2}, \frac{5n-4+\sqrt{9n^2+8n-16}}{4(n-2)} \Big\},
\end{align*}
which is reduced 
\begin{equation}\label{c27}
\max\Big\{2,\frac{n+4}{n},\frac12+\frac1{n-2-2\kappa}\Big\} < p < \frac{5n-4+\sqrt{9n^2+8n-16}}{4(n-2)}
\end{equation}
by using $p>2$, $n\ge3$ and combining the first condition in \eqref{c26}.

The assumption \eqref{ass2} follows from inserting $p=1+\frac{2-2\tau+\alpha}{n-2}$ into \eqref{c27}. In fact, \eqref{c27} is expressed with respect to $\tau$, as follows:
\begin{equation}
\frac{\alpha}{2}-\frac{n-4+\sqrt{9n^2+8n-16}}{8}<\tau < \frac{\alpha}{2}-\max\{\frac{n-4}{2}, \frac{n-4}{n},\frac{\kappa}{n-2-2\kappa}-\frac{n}{4}\}.
\end{equation}
Finally, we make the lower bound of $\tau$ less than the upper ones of $\tau$ to deduce 
\begin{equation*}
2\kappa < \frac{5n-4-\sqrt{9n^2+8n-16}}{2},
\end{equation*} 
which implies the assumption \eqref{ass1}.
Indeed, to obtain \eqref{ass1}, we can compute as follows:
\begin{eqnarray*}
&&\frac{n-4+\sqrt{9n^2+8n-16}}{8}>\frac{\kappa}{n-2-2\kappa}-\frac{n}{4}\\
&\Leftrightarrow&\frac{n-2-2\kappa}{8}>\frac{\kappa}{3n-4+\sqrt{9n^2+8n-16}}\times\frac{\sqrt{9n^2+8n-16}-(3n-4)}{\sqrt{9n^2+8n-16}-(3n-4)}\\
&\Leftrightarrow&n-2>\frac{\sqrt{9n^2+8n-16}+5n-4}{4(n-1)}\kappa.
\end{eqnarray*}
This is equivalent to
\begin{align*}
\kappa<\frac{4(n-1)(n-2)}{5n-4+\sqrt{9n^2+8n-16}}&=\frac{4(n-1)(n-2)\big\{5n-4-\sqrt{9n^2+8n-16}\big\}}{(5n-4)^2-9n^2-8n+16}\\
&=\frac{5n-4-\sqrt{9n^2+8n-16}}{4}.
\end{align*}
This ends the proof.
\end{proof}

\subsection{Local well-posedness in the energy space}
By Duhamel's principle, we first write the solution of the Cauchy problem \eqref{S} as fix points of the function
$$\Phi(u)=e^{-it\mathcal{K}_{\lambda}}u_0 + i\epsilon\int_0^t e^{-i(t-s)\mathcal{K}_{\lambda}} \mathcal N[u](s,\cdot) ds$$
where $\mathcal N[u]=|x|^{-\tau}|u|^{p-2}(J_\alpha \ast |\cdot|^{-\tau}|u|^p)u$. 
For appropriate values of $T,M,N>0$, we shall show that $\Phi$ defines a contraction map on 
$$X(T,M,N)=\{u \in C_t(I;H_\lambda^1) \cap L_t^{q}(I;W^{1,r}_{\lambda}): \sup_{t\in I} \|u\|_{H_{\lambda}^1}\leq M, \|u\|_{\mathcal W_{\lambda}(I)}\leq N\}$$
equipped with the distance 
$$d(u,v)=\|u-v\|_{\Lambda(I)}.$$
Here, $I=[0,T]$ and $(q,r)$ is given as in Proposition \ref{str}.
We also define 
$$\|u\|_{\mathcal W_{\lambda}(I)}:= \|u\|_{\Lambda(I)} + \|\sqrt{\mathcal K_{\lambda}}u\|_{\Lambda(I)}$$
and 
$$\|u\|_{\mathcal W{'}_{\lambda}(I)}:= \|u\|_{\Lambda'(I)} + \|\sqrt{\mathcal K_{\lambda}}u\|_{\Lambda'(I)}.$$

We now show that $\Phi$ is well defined on $X$. By Proposition \ref{str}, we get 
\begin{equation}\label{w1}
\|\Phi(u)\|_{\mathcal W_{\lambda}(I)}\leq C\|e^{-it\mathcal K_{\lambda}}u_0\|_{\mathcal W_{\lambda}(I)} +C\big\|\mathcal N[u]\big\|_{\mathcal W{'}_{\lambda}(I)}
\end{equation}
and 
\begin{equation*}
\sup_{t\in I}\|\Phi(u)\|_{H_{\lambda}^1}\leq \|u_0\|_{H_{\lambda}^1}+\sup_{t \in I}\Big\|\int_0^t e^{-i(t-s)\mathcal{K}_{\lambda}} \mathcal N[u](\cdot,s) ds\Big\|_{H_{\lambda}^1}.
\end{equation*}
Here, for the second inequality we used the fact that $e^{it\mathcal K_{\lambda}}$ is an unitary on $L^2$.
Since $\|\langle \sqrt{\mathcal K_{\lambda}} \rangle u \|\lesssim \|u\| + \|\sqrt{\mathcal K_{\lambda}}u\|$, using the fact $e^{it\mathcal K_{\lambda}}$ is an unitary on $L^2$ again, and then applying the dual estimate of the first one in Proposition \ref{str}, we see that
$$ \sup_{t\in I} \Big\| \int_0^t e^{-i(t-s)\mathcal{K}_{\lambda}} \mathcal N[u](\cdot,s) ds \Big\|_{H_{\lambda}^1} \lesssim \|\mathcal{N}[u]\|_{\Lambda'(I)} +\|\sqrt{\mathcal{K}_{\lambda}}\mathcal{N}[u]\|_{\Lambda'(I)}. $$
Hence, 
\begin{equation*}
\sup_{t\in I}\|\Phi(u)\|_{H_\lambda^1} \leq C \|u_0\|_{H_{\lambda}^1}+C\|\mathcal{N}[u]\|_{\mathcal{W}_{\lambda}'(I)}.
\end{equation*}
On the other hand, using Lemma \ref{non}, we get 
\begin{align}\label{w2}
\nonumber
\|\mathcal{N}[u]\|_{\mathcal{W}_{\lambda}'(I)}&\leq C \|\sqrt{\mathcal K_{\lambda}}u\|_{\Lambda(I)}^{2p-1} + \|\sqrt{\mathcal K_{\lambda}}u\|_{\Lambda(I)}^{2p-2}\|u\|_{\Lambda(I)} \\
\nonumber
&\leq C \|\sqrt{\mathcal K_{\lambda}}u\|_{\Lambda(I)}^{2p-2}\|u\|_{\mathcal{W}_{\lambda}(I)}\\
&\leq C N^{2p-1}
\end{align}
if $u \in X$, and for some $\varepsilon>0$ small enough which will be chosen later we get
\begin{equation}\label{sm}
\|e^{it\mathcal{K}_{\lambda}}u_0\|_{\mathcal{W}_{\lambda}(I)}\leq \varepsilon
\end{equation}
which holds for a sufficiently small $T>0$ by the dominated convergence theorem. 
We now conclude that 
\begin{equation*}
\|\Phi(u)\|_{\mathcal{W}_{\lambda}(I)} \leq \varepsilon + CN^{2p-1} \quad \textnormal{and} \quad \sup_{t \in I}\|\Phi(u)\|_{H_{\lambda}^1} \leq C \|u_0\|_{H_\lambda^1} + CN^{2p-1}.
\end{equation*}
Hence we get $\Phi(u)\in X$ for $u \in X$ if 
\begin{equation}\label{w3}
\varepsilon + CN^{2p-1} \leq N \quad \textnormal \quad C\|u_0\|_{H_{\lambda}^1} + CN^{2p-1} \leq M.
\end{equation}

Next we show that $\Phi$ is a contraction on $X$. 
Using the same argument used in \eqref{w1}, we see
\begin{equation*}
\|\Phi(u)-\Phi(v)\|_{\Lambda(I)} \leq C \|\mathcal{N}[u]-\mathcal{N}[v]\|_{\Lambda'(I)}.
\end{equation*}
By applying Lemma \ref{non} (see \eqref{non2}), we see 
\begin{align*}
\|\mathcal{N}[u]-\mathcal{N}[v]\|_{\Lambda'(I)}&\leq C\big(\|\sqrt{\mathcal{K}_{\lambda}} u\|_{\Lambda(I)}^{2p-1}+\|\sqrt{\mathcal{K}_{\lambda}} v\|_{\Lambda(I)}^{2p-2}\big)\|u-v\|_{\Lambda(I)} \\
&\leq C N^{2p-2}\|u-v\|_{\Lambda(I)}
\end{align*}
as in \eqref{w2}.
Hence, for $u,v \in X$ we obtain $d(\Phi(u), \Phi(v))\leq C N^{2p-2}d(u,v)$.
Now by taking $M=2C\|u_0\|_{H_\lambda^1}$ and $N=2\varepsilon$ and then choosing $\varepsilon>0$ small enough so that \eqref{w3} holds and $CN^{2p-2}\leq 1/2$, it follows that $\Phi$ is a contraction on $X$.
Therefore, we have proved that there exists a unique local solution with $u \in C_t(I;H_{\lambda}^1) \cap L_t^q(I;W_{\lambda}^{1,r})$ for any admissible pair $(q,r)$.

\subsection{Global well-posedness in the energy space for small data}
Using the first estimate in Proposition \ref{str}, we observe that \eqref{sm} is satisfied also if $\|u_0\|_{H_{\lambda}^1}$ is sufficiently small,
\begin{equation*}
\|e^{-it\mathcal K_{\lambda}}u_0\|_{\mathcal{W}_\lambda(I)} \leq C \|u_0\|_{H_{\lambda}^1}\leq \varepsilon
\end{equation*}
from which one can take $T=\infty$ in the above argument to obtain a global unique solution. 

The continuous dependence of the solution $u$ with respect to the initial data $u_0$ follows clearly in the same way:
\begin{align*}
d(u,v) &\lesssim d(e^{-it\mathcal{K}_{\lambda}}u_0,e^{-it\mathcal{K}_{\lambda}}v_0) + d\Big(\int_0^t e^{-i(t-s)\mathcal{K}_{\lambda}}\mathcal{N}[u]ds,\int_0^t e^{-i(t-s)\mathcal{K}_{\lambda}}\mathcal{N}[v]ds \Big) \\
&\lesssim \|u_0 - v_0\| + \frac12 d(u,v)
\end{align*}
which implies 
\begin{align*}
d(u,v) \lesssim \|u_0-v_0\|_{H_{\lambda}^1}.
\end{align*}
Here, $u,v$ are the corresponding solutions for initial data $u_0, v_0$, respectively.
\subsection{Scattering in the energy space for small data}
To prove the scattering property, we first note that 
\begin{align*}
\|e^{it_2 \mathcal{K}_{\lambda}}u(t_2)-e^{it_2 \mathcal{K}_{\lambda}}u(t_1)\|_{H_{\lambda}^1} = \Big\|\int_{t_1}^{t_2} e^{is\mathcal{K}_{\lambda}}\mathcal{N}[u]\Big\|_{H_{\lambda}^1}\\
\lesssim \|\mathcal{N}[u]\|_{\mathcal{W}_{\lambda}{'}([t_1 ,t_2])} \\
\lesssim \|u\|^{2p-1}_{\mathcal{W}_{\lambda}([t_1,t_2])} \quad \rightarrow\quad 0
\end{align*}
as $t_1, t_2 \rightarrow {\infty}$. 
This implies that $\phi :=\lim_{t\rightarrow {\infty}} e^{it\mathcal{K}_{\lambda}}u(t)$ exists in $H_\lambda^1$.
Furthermore, 
$$u(t)-e^{-it\mathcal{K}_{\lambda}}\phi=i\int_t^{\infty} e^{i(t-s)\mathcal{K}_{\lambda}}\mathcal{N}[u]ds,$$
and hence 
\begin{align*}
\|u(t)-e^{-it{\mathcal K}_\lambda}\phi\|_{H^1_\lambda}&=\Big\|\int_t^{\infty} e^{i(t-s)\mathcal{K}_{\lambda}}\mathcal{N}[u] ds\Big\|_{H_{\lambda}^1}\\
&\lesssim\|\mathcal{N}[u]\|_{\mathcal{W}_{\lambda}'([t,\infty])} \\
&\lesssim \|u\|^{2p-1}_{\mathcal{W}_{\lambda}([t,\infty])} \quad \rightarrow \quad 0
\end{align*}
as $t\rightarrow {\infty}$. 
{The scattering is proved.}
\section{Ground states and Gagliardo-Nirenberg estimate}
In this section, we prove Theorem \ref{gag} dealing with the existence of ground states solutions to \eqref{E} and the Gagliardo-Nirenberg type estimate \eqref{gagg}.
\subsection{Gagliardo-Nirenberg estimate}
Using the Hardy-Littlewood-Sobolev inequality (Lemma \ref{hls}), we first see
\begin{equation}\label{ie}
\int_{\mathbb{R}^n} |x|^{-\tau}|u|^p (I_\alpha \ast |\cdot|^{-\tau}|u|^p) dx \lesssim \big\||x|^{-\tau}|u|^p\big\|^2_{{\frac{2n}{\alpha+n}}}
\end{equation}
if $0<\alpha<n.$
Applying Lemma \ref{ckn} to the right-hand side of \eqref{ie} with $b=-\frac{\tau}{p}$, $q=\frac{2np}{\alpha+n}$, $a=0$ and $p=2$, we get
\begin{equation}\label{5.100'}
\big\||x|^{-\tau}|u|^p\big\|^2_{{\frac{2n}{\alpha+n}}}= \big\||x|^{-\frac{\tau}{p}}u\big\|^{2p}_{{\frac{2np}{\alpha+n}}} \lesssim \|\nabla u\|^{2p}
\end{equation}
if 
\begin{equation}\label{iee}
	0<\frac{\alpha+n}{2np} \leq \frac12 <1, \quad -\frac{\alpha+n}{2p}<-\frac{\tau}{p} \leq 0, \quad \frac{\tau}{p}-1=\frac{\alpha+n}{2p}-\frac{n}2.
\end{equation}
Finally, using the equivalent norm to Sobolev one (see Lemma \ref{2.2}), we obtain the desired estimate \eqref{inte} if $\frac{1+\kappa}{n}<\frac12<\min\{1, 1-\frac{\kappa}{n}\}$ which does not affect the assumptions in \eqref{as}.

Now it remains to derive the assumptions in \eqref{as}.  
We note that the last equality in \eqref{iee} is equivalent to $p=1+\frac{2-2\tau+\alpha}{n-2}$.
Using $p=1+\frac{2-2\tau+\alpha}{n-2}$, the requirements \eqref{iee} can be written as
\begin{equation}\label{ie1}
	\alpha+n>0, \quad \alpha+n \ge n\tau, \quad 0\leq 2\tau<\alpha+n,
\end{equation}
which are reduced to $0<\tau\leq 1+\frac{\alpha}{n}$ since $\alpha+n\geq n\tau>2\tau>0$, as desired.

\subsection{Existence of ground states}
By using \eqref{gagg}, we first set $J(u)=\|\sqrt{\mathcal K_\lambda}u\|^{2p}/\|\mathcal P[u]\|$ and take a sequence $\{u_n\}_{n\in\mathbb{N}}$ in $H_{\lambda}^1$ such that 
\begin{align*}
	\gamma := \frac1{C_{n,\tau,\alpha,\lambda}}
	=\lim_{n\rightarrow \infty} \frac{\|\sqrt{\mathcal{K}_{\lambda}}u_n\|^{2p}}{\mathcal{P}[u_n]}.
\end{align*}
By the scaling $u(x) \mapsto u^{\delta,\mu}(x)=\delta u(\mu x)$ for $\delta, \mu \in \mathbb{R}$, we have
\begin{align*}
	\|u^{\delta,\mu}\|^2&=\delta^2 \mu^{-n}\|u\|^2 \\
	\|\sqrt{\mathcal K_{\lambda}}u^{\delta,\mu}\|^2&= \|\nabla u^{\delta,\mu}\|^2+\lambda\big\|\frac{u^{\delta,\mu}}{|x|}\big\|^2	\\
	&=\delta^2 \mu^{2-n}\Big(\|\nabla u\|^2+\lambda\big\|\frac{u}{|x|}\big\|^2\Big)=\delta^2 \mu^{2-n}\|\sqrt{\mathcal K_\lambda}u\|^2 \\
	\int_{\mathbb{R}^n}|x|^{-\tau}|u^{\delta,\mu}|^p(I_\alpha \ast |\cdot|^{-\tau}|u^{\delta,\mu}|^p) dx &= \delta^{2p}\mu^{2\tau-n-\alpha}\int_{\mathbb{R}^n}|x|^{-\tau}|u|^p(I_\alpha \ast |\cdot|^{-\tau}|u|^p) dx,
\end{align*}
which implies that $J(u^{\delta,\mu})=J(u)$ by $p=1+\frac{2-2\tau+\alpha}{n-2}$.
Let $\psi_n = u_n^{\delta_n,\mu_n}$ where 
$$\delta_n= \frac{\|u_n\|^{\frac{n}2-1}}{\|\sqrt{\mathcal K_\lambda} u_n\|^{\frac{n}{2}}}, \quad \mu_n=\frac{\|u_n\|}{\|\sqrt{\mathcal K_\lambda} u_n\|}.$$
Then, we have
$$\|\psi_n\|=\|\sqrt{\mathcal K_\lambda} \psi_n\|=1 \quad \textnormal{and} \quad \gamma=\lim_{n\rightarrow\infty} J(\psi_n)=\lim_{n\rightarrow \infty} \frac1{\mathcal P[\psi_n]}.$$ 

Now we take $\psi \in H_\lambda^1$ so that $\psi_n \rightharpoonup \psi$ in $H_\lambda^1$ and we will show that 
$$\frac{1}{\mathcal P[\psi_n]} \rightarrow \frac1{\mathcal P[\psi]} \quad \textnormal{as} \quad n\rightarrow \infty.$$
By using Lemma \ref{hls} {via \eqref{5.100'}}, we have 
\begin{align}
\nonumber
&\int_{\mathbb{R}^n} |x|^{-\tau}|\psi_n|^p (I_\alpha \ast |\cdot|^{-\tau}|\psi_n|^p)- |x|^{-\tau}|\psi|^p (I_\alpha \ast |\cdot|^{-\tau}|\psi|^p)dx \\
\nonumber
&\quad =\int_{\mathbb{R}^n} |x|^{-\tau}|\psi|^p \big(I_\alpha \ast |\cdot|^{-\tau}(|\psi_n|^p - |\psi|^p)\big) dx\\
\nonumber
&\qquad \qquad \qquad+ \int_{\mathbb{R}^n}|x|^{-\tau}(|\psi_n|^p-|\psi|^p)(I_\alpha \ast |\cdot|^{-\tau}|\psi_n|^p)dx\\
\nonumber
&\quad\lesssim \Big(\big\||x|^{-\tau}|\psi|^p\big\|_{{\frac{2n}{\alpha+n}}} +\big\||x|^{-\tau}|\psi_n|^p\big\|_{{\frac{2n}{\alpha+n}}}\Big)
\big\||x|^{-\tau}(|\psi_n|^p-|\psi|^p)\big\|_{{\frac{2n}{\alpha+n}}}\\
\label{dif}
&\quad\lesssim \big\||x|^{-\tau}(|\psi_n|^p-|\psi|^p)\big\|_{{\frac{2n}{\alpha+n}}}.
\end{align}
Using the following simple inequality 
$$|u|^p-|v|^p \lesssim |u-v|(|u|^{p-1}+|v|^{p-1}), \quad p\ge1$$
and H\"older's inequality, the last term in \eqref{dif} is bounded as 
\begin{align}\label{i}
\nonumber
\big\||x|^{-\tau}(|\psi_n|^p-|\psi|^p)\big\|_{{\frac{2n}{\alpha+n}}} 
&\lesssim (\|\psi_n\|_{{(p-1)a_1}}^{p-1}+\|\psi\|_{{(p-1)a_1}}^{p-1}) \||x|^{-\tau}|\psi-\psi_n|\|_{{a_2}}\\
&\lesssim (\|\psi\|_{H_\lambda^1}^{p-1}+\|\psi_n\|_{H_\lambda^1}^{p-1})\||x|^{-\tau}|\psi-\psi_n|\|_{{a_2}}
\end{align}
if $0<\tau<2$ and 
\begin{equation}\label{gagc}
	\frac{\alpha+n}{2n}=\frac{1}{a_1}+\frac{1}{a_2}, \quad \frac{n-2}{2n}\leq\frac{1}{(p-1)a_1}\leq\frac12.
\end{equation}
{Indeed, for the last inequality we used the Sobolev embedding, $H^1(\mathbb{R}^n) \hookrightarrow L^q({\mathbb{R}^n})$ for $2\leq q \leq \frac{2n}{n-2}$ if $n\ge3$.}
Thanks to the compactness of the Sobolev injection, Lemma \ref{compact}, under the condition
\begin{equation}\label{gagcc}
	\frac{n-2}{2(n-\tau)}<\frac{1}{a_2}<\frac12,
\end{equation}
we then get $1/{\mathcal P[\psi_n]} \rightarrow 1/{\mathcal P[\psi]}=\gamma$ as $n \rightarrow \infty$. 
We need to check that there exist $a_1$ and $a_2$ satisfying \eqref{gagc}, \eqref{gagcc} and the assumptions in Theorem \ref{gag}, but we will postpone this until the end of the proof.

By the lower semi-continuity of the norm, we see
$$\|\psi\|\leq 1 \quad \textnormal{and} \quad \|\sqrt{\mathcal K_\lambda} \psi\|\leq 1,$$
from which $J(\psi)<\gamma$, and hence $\|\psi\|=\|\sqrt{\mathcal K_\lambda}\psi\|=1.$
Consequently, 
$$\psi_n \rightarrow \psi \quad \textnormal{in} \quad H_\lambda^1 \quad \textnormal{and} \quad \gamma=J(\psi)=\frac1{\mathcal P[\psi]}.$$
$\psi$ satisfies \eqref{E} because the minimizer satisfies the Euler equation 
$$\partial_\epsilon J(\psi+\epsilon \eta)_{|\epsilon=0}=0, \quad \forall \eta \in C_0^{\infty} \cap H_\lambda^1.$$

It remains to check the existence of $a_1$ and $a_2$ satisfying the conditions \eqref{gagc}, \eqref{gagcc} under the assumptions in Theorem \ref{gag}.
Substituting the first condition in \eqref{gagc} into the second one of \eqref{gagc} with $p=1+\frac{2-2\tau+\alpha}{n-2}$, we see
{\begin{equation}\label{gagc2}
	\frac{\alpha+n}{2n}-\frac{2-2\tau+\alpha}{2(n-2)}\leq\frac1{a_2} \leq \frac{\alpha+n}{2n}-\frac{2-2\tau+\alpha}{2n}.
	\end{equation}}
To eliminate $a_2$, we make the lower bounds of $1/a_2$ of \eqref{gagcc} and \eqref{gagc2} less than the upper ones of $1/a_2$ of \eqref{gagcc} and \eqref{gagc2}.
{Indeed, starting the process from the lower bound in \eqref{gagc2}, we arrive at $n\tau<\alpha+n$ which is satisfied by the assumption \eqref{as}.
Similarly from the lower bound in \eqref{gagcc}, we arrive at $0<\tau<\frac{n+2}{2}$, but this is eliminated by \eqref{as} using the facts that $n\ge3$ and $\tau<2$.}
\section{Blow-up of the energy solutions}
In this section, we prove Theorem \ref{t1} which provides a criterion for blow-up phenomena in the energy-critical focusing regime under the threshold of the ground state. As a consequence, we establish Corollary \ref{t2}. Moreover, we prove Proposition \ref{s} and Corollary \ref{s2} about energy bounded solutions.
\subsection{Criterion for blow-up}
In order to prove Theorem \ref{t1}, we use proof by contradiction through the following inequality which will be proved:
\begin{equation}\label{qq}
	V_{R}''\leq4\,\mathcal I[u]+\frac C{R^{2\tau}}+\frac C{R^2},
\end{equation}
where $\mathcal I [u] = \|\sqrt{\mathcal K_\lambda} u\|^2 - \mathcal P[u]$ and $R\gg1$.
Indeed, taking $u_0\in H^1_\lambda$ with \eqref{ss'} and assuming that $u$ is global, for $R\gg1$,
$$V_{R}''\leq4\,\mathcal I[u]+\frac C{R^{2\tau}}+\frac C{R^2}<-c<0$$
if there is no sequence $t_n\to\infty$ such that $\|\sqrt{\mathcal K_\lambda} u(t_n)\|\to \infty$, which is contradiction.

Before starting to prove \eqref{qq}, we first define $\phi_R(\cdot):=R^2\phi(\frac{\cdot}{R})$, $R>0$, where the radial function $\phi\in C_0^\infty(\R^n)$ satisfies
$$\phi(|x|)=\phi(r):=\left\{
\begin{array}{ll}
	\frac{r^2}2,\quad\mbox{if}\quad r\leq1 ;\\
	0,\quad\mbox{if}\quad r\geq2,
\end{array}
\right.\quad\mbox{and}\quad \phi''\leq1.$$
Then, $\phi_R$ satisfies
$$\phi_R''\leq1,\quad \phi_R'(r)\leq r,\quad\Delta \phi_R\leq N$$
and, for $|x|\leq R$
\begin{align}\label{calc}
	\nabla\phi_R(x)=x,\quad\Delta\phi_R(x)=N.     
\end{align}

By recalling the definition of $V(t)$ and $M(t)$ in Section 2, we denote the localized variance and Morawetz action as
\begin{align*}
	V_R(t):=\int_{\R^n}\phi_R(x)|u(x,\cdot)|^2\,dx, \quad 	V_R'(t)=M_R(t):=2\Im\int_{\R^n}\bar u\nabla \phi_R \cdot \nabla udx.
\end{align*}
By Proposition \ref{mrwz}, we divide $M_R'$ into two parts, $A$ and $B$, as $M_R'(t)=A + B$
where
$$A=4\sum_{k,l=1}^{N}\int_{\R^n}\partial_l\partial_k\phi_R\Re(\partial_ku\partial_l\bar u)dx-\int_{\R^n}\Delta^2\phi_R|u|^2dx+4\lambda\int_{\R^n}\nabla\phi_R\cdot x\frac{|u|^2}{|x|^4}dx$$
and 
\begin{align}
	\nonumber
	B&=-\frac{2(p-2)}{p}\int_{\R^n}\Delta\phi_R|x|^{-\tau}|u|^p(I_\alpha*|\cdot|^{-\tau}|u|^{p})dx\\
	\nonumber
	&\quad\qquad -\frac{4\tau}p\int_{\R^n}x\cdot\nabla\phi_R|x|^{-\tau-2}|u|^{p}(I_\alpha*|\cdot|^{-\tau}|u|^p)dx\\
	\nonumber
	&\qquad\qquad\qquad -\frac{4(N-\alpha)}p\sum_{k=1}^N\int_{\R^n}|x|^{-\tau}|u|^{p}\partial_k\phi_R(\frac{x_k}{|\cdot|^2}I_\alpha*|\cdot|^{-\tau}|u|^p)dx\\
	\label{123}
	&=:B_1 + B_2+ B_3.
\end{align}

Using the following radial relations
\begin{equation}\label{''}
	\partial_k=\frac{x_k}r\partial_r,\quad\partial_l\partial_k=\Big(\frac{\delta_{lk}}r-\frac{x_lx_k}{r^3}\Big)\partial_r+\frac{x_lx_k}{r^2}\partial_r^2
\end{equation}
and the Cauchy-Schwarz inequality via the properties of $\phi$,
it follows that
\begin{align}
	\nonumber
	A&= 4\int_{\mathbb{R}^N} |\nabla u|^2 \frac{\phi_R'}{r}dx + 4 \int_{\mathbb{R}^N}
	|x\cdot \nabla u|^2 \big(\frac{\phi_R''}{r^2}-\frac{\phi_R'}{r^3}\big) dx \\
	\nonumber
	&\qquad \qquad\qquad \qquad\qquad \qquad\quad-\int_{\mathbb{R}^N} \Delta^2\phi_R|u|^2dx + 4 \int_{\mathbb{R}^N} \frac{|u|^2}{r^3}\phi_R' dx \\
	\nonumber
	&\leq4\int_{\R^n}|\nabla u|^2\frac{\phi_R'}r\,dx+4\int_{\R^n}\frac{|x\cdot\nabla u|^2}{r^2}\big(1-\frac{\phi_R'}r\big)dx\\
	\nonumber
	&\qquad \qquad \qquad \qquad\qquad \qquad\quad -\int_{\R^n}\Delta^2\phi_R|u|^2\,dx+4\lambda\int_{\R^n}\frac{|u|^2}{r^3}\phi_R'dx\nonumber\\
	\label{(I)}
	&\leq4\int_{\R^n}|\nabla u|^2dx-\int_{\R^n}\Delta^2\phi_R|u|^2dx+4\lambda\int_{\R^n}\frac{|u|^2}{r^2}dx.
\end{align}

On the other hand, to handle the part $B$, we split the integrals in $B$ into the regions $|x|<R$ and $|x|>R$.
Then, by \eqref{calc}, the first two terms in $B$ are written
\begin{align}
	\nonumber
	B_1+B_2&=\frac{2N(2-p)-4\tau}{p}\int_{|x|<R}|x|^{-\tau}|u|^p(I_\alpha*|\cdot|^{-\tau}|u|^{p})dx\\
	\nonumber
	&\qquad \qquad \qquad \qquad +O\bigg(\int_{|x|>R}|x|^{-\tau}|u|^{p}(I_\alpha*|\cdot|^{-\tau}|u|^p)dx\bigg)\\
	&=\frac{2(N(2-p)-2\tau)}p\mathcal{P}[u]+O\bigg(\int_{|x|>R}|x|^{-\tau}|u|^{p}(I_\alpha*|\cdot|^{-\tau}|u|^p)dx\bigg).\label{12}
\end{align}
For the third term $B_3$, with calculus done in \cite[Lemma 4.5]{st4}, we have
\begin{align}
	B_3&=\frac{2(\alpha-N)}{p}\int_{|y|<R}\int_{|x|<R}I_\alpha(x-y)|y|^{-\tau}|u(y)|^p|x|^{-\tau}|u(x)|^{p}\,dx\,dy\nonumber\\
	&\qquad \qquad \qquad \qquad \qquad \qquad +O\bigg(\int_{|x|>R}(I_\alpha*|\cdot|^{-\tau}|u|^p)|x|^{-\tau}|u|^pdx\bigg)\nonumber\\
	&=\frac{2(\alpha-N)}{p}\int_{|x|<R}|x|^{-\tau}|u(x)|^{p}(I_\alpha*|\cdot|^{-\tau}|u|^p)dx \nonumber\\
	&\qquad \qquad \qquad \qquad \qquad \qquad+O\bigg(\int_{|x|>R}(I_\alpha*|\cdot|^{-\tau}|u|^p)|x|^{-\tau}|u|^pdx\bigg)\nonumber\\
	&=\frac{2(\alpha-N)}{p}\mathcal P[u]+O\bigg(\int_{|x|>R}|x|^{-\tau}|u|^p(I_\alpha*|\cdot|^{-\tau}|u|^p)dx\bigg).\label{372}
\end{align} 
Combining \eqref{(I)}, \eqref{123}, \eqref{12} and \eqref{372}, we then obtain
\begin{align}
	M_R'&\leq-\int_{\R^n}\Delta^2\phi_R|u|^2\,dx+4\int_{\R^n}|\nabla u|^2+4\lambda\int_{\R^n}\frac{|u|^2}{r^2}\,dx
	-4\mathcal P[u]\nonumber\\
	&\qquad \qquad \qquad \qquad \qquad \qquad \quad  +O\left(\int_{|x|>R}|x|^{-\tau}|u|^{p}(I_\alpha*|\cdot|^{-\tau}|u|^p)\,dx\right)\nonumber\\
	&\leq4\Big(\|\sqrt{\mathcal K_\lambda} u\|^2-\mathcal P[u]\Big)+O\left(\int_{|x|>R}|x|^{-\tau}|u|^{p}(I_\alpha*|\cdot|^{-\tau}|u|^p)\,dx\right)+O(R^{-2}).\nonumber
\end{align}
Here, for the last inequality we used the fact that $|\partial^{\nu}\phi_R|\lesssim R^{2-|\nu|}$.

Now, using Lemma \ref{hls} and Lemma \ref{ckn} with $b=-\frac{\tau}{p}$, $q=\frac{2Np}{\alpha+N}$, $a=0$ and $p=2$, we obtain
\begin{eqnarray*}
	\int_{|x|>R}|x|^{-\tau}|u|^{p}(I_\alpha*|\cdot|^{-\tau}|u|^p)\,dx 
	&\lesssim& \||x|^{-\tau}|u|^{p}\|^2_{\frac{2N}{\alpha+N}}\\
	&\lesssim& R^{-2\tau}\|u\|_{\frac{2Np}{\alpha+N}}^{2p}\\
	&\lesssim& R^{-2\tau}\|\nabla u\|^{2p}
\end{eqnarray*}
if 
\begin{equation}\label{q}
	0<\frac{\alpha+n}{2np}\leq \frac12<1 , \quad -\frac{\alpha+n}{2p}<-\frac{\tau}{p}\leq 0, \quad \frac{\tau}{p}-1=\frac{\alpha+n}{2p}-\frac{n}{2}.
\end{equation}
Here we note that the last equality in \eqref{q} is equivalent to $p=1+\frac{2-2\tau+\alpha}{n-2}$.
Using $p=1+\frac{2-2\tau+\alpha}{n-2}$, the requirement \eqref{q} can be written as \eqref{ie1}
which are reduced to $0<\tau\leq 1+\frac{\alpha}{n}$ since $\alpha+n\geq n\tau>2\tau>0$.

Consequently, for large $R\gg1$, using the equivalent norm to Sobolev one (see Lemma \ref{2.2}), we get
\begin{equation*}
	M_R'\leq4\,\mathcal I[u]+\frac C{R^{2\tau}}\|\sqrt{\mathcal K_\lambda} u\|^{2p}+\frac C{R^2},
\end{equation*}
as desired.

\subsection{The boundedness of the energy solution}
Now, we prove Proposition \ref{s}.
Specifically, the energy bound is demonstrated by combining the conservation law with the following lemma known as coercivity (or energy trapping) results, obtained through the assumption \eqref{ss}.
\begin{lem}\label{bnd'}
Let $\varphi\in H_\lambda^1$ be a ground state solution to \eqref{E}.
Assume that there is $0<c<1$ satisfying
$$\mathcal P[u]<c \mathcal P[\varphi], \quad u\in H_\lambda^1.$$
Then there exists a constant $c_\varphi>0$ such that 
\begin{align*}
\|\sqrt{\mathcal K_\lambda} u\|^2&<c_\varphi\mathcal E[u].
\end{align*}
\end{lem}
\begin{proof}
Thanks to \eqref{E}, we first see
\begin{equation}\label{p}
\mathcal P[\varphi]:=\int_{\R^n} |x|^{-\tau}|\varphi|^p(I_\alpha *|\cdot|^{-\tau}|\varphi|^p)dx=\|\sqrt{\mathcal K_\lambda}\varphi\|^2.
\end{equation}

Applying the Gagliardo-Nirenberg inequality, Theorem \ref{gag}, we have
\begin{align*}
(\mathcal P[u])^p \leq (\mathcal P[u])^{p-1}\cdot C_{N,\tau,\alpha,\lambda} \|\sqrt{\mathcal K_\lambda} u\|^{2p}\leq \Big(\frac{\mathcal P[u]}{\mathcal P[\varphi]}\Big)^{p-1}\|\sqrt{\mathcal K_\lambda} u\|^{2p}.
\end{align*}
Here, for the second inequality, we used the fact that
\begin{equation}\label{poh}
	C_{N,\tau,\alpha,\lambda}=\frac{\mathcal P(\varphi)}{\|\sqrt{\mathcal K_\lambda}\varphi\|^{2p}}=({\mathcal P[\varphi]})^{1-p}
\end{equation}
by using \eqref{p}.
Therefore, we obtain
\begin{eqnarray*}
\mathcal P[u]
&\leq&\Big(\frac{\mathcal P[u]}{\mathcal P[\varphi]}\Big)^\frac{p-1}p\|\sqrt{\mathcal K_\lambda} u\|^2,
\end{eqnarray*}
which implies that
\begin{align*}
\mathcal E[u]&=\|\sqrt{\mathcal K_\lambda} u\|^2-\frac1p\mathcal P[u]\\
&\geq\Big(1-\frac{1}p\Big(\frac{\mathcal P[u]}{\mathcal P[\varphi]}\Big)^\frac{p-1}p\Big)\|\sqrt{\mathcal K_\lambda} u\|^2\\
&\geq\Big(1-\frac{c^{(p-1)/p}}p\Big)\|\sqrt{\mathcal K_\lambda} u\|^2.
\end{align*}
{This concludes the proof.}
\end{proof}

\subsection{Energy bounded/non-global solutions}
Finally, we prove Corollaries \ref{t2} and \ref{s2}, which presents the dichotomy of energy bounded/non-global existence of solutions.
\subsubsection{Energy bounded solutions}
First, Corollary \ref{s2} follows from the invariance of \eqref{t11} and \eqref{t12} under the flow of \eqref{S}.
We first define a function $f:[0,T^\ast) \rightarrow \mathbb{R}$ as
\begin{equation}\label{def}
	f(t)= t- \frac{C_{N,\tau,\alpha,\lambda}}{p} t^{p}.
\end{equation}
Since $p>1$, the function $f(t)$ has a maximum value $f(t_1)=\frac{p-1}{p}C_{N,\tau,\alpha,\lambda}^{-\frac1{p-1}}$ at $t_1= (C_{N,\tau,\alpha,\lambda})^{-\frac1{p-1}}$.
We note here that $t_1=\|\sqrt{\mathcal K_\lambda} \varphi\|^2$ by \eqref{poh} and \eqref{p}.

Using the Gagliardo-Nirenberg type inequality, Theorem \ref{gag}, we see 
\begin{align}
	\mathcal E[u]=\|\sqrt{\mathcal K_\lambda} u\|^2 - \frac1p \mathcal P[u] &\geq\|\sqrt{\mathcal K_\lambda} u\|^2-\frac{C_{N,\tau,\alpha,\lambda}}{p}\|\sqrt{\mathcal K_\lambda} u\|^{2p}\label{xxx}\\
	&=f\big(\|\sqrt{\mathcal K_\lambda} u\|^2\big).\nonumber
\end{align}
By the assumption \eqref{t11} with \eqref{poh} and \eqref{p}, we also see
\begin{equation}\label{x}
	\mathcal E[u_0]<\mathcal E[\varphi] =f(t_1),
\end{equation}
which implies 
\begin{equation}\label{xx}
	f\big(\|\sqrt{\mathcal K_\lambda}u\|^2\big) \leq \mathcal E[u]= \mathcal E[u_0] <f(t_1).
\end{equation}
Since $\|\sqrt{\mathcal K_\lambda}u_0\|^2<\|\sqrt{\mathcal K_\lambda}\varphi\|^2=t_1$ by the assumption \eqref{t12}, and the continuity in time with \eqref{xx}, we get
$$\|\sqrt{\mathcal K_\lambda}u(t)\|^2 < t_1, \quad \forall t \in [0,T^\ast),$$
which is equivalent to $$\mathcal{MG}[u(t)]<1, \quad \forall t \in [0,T^\ast).$$
Therefore \eqref{t11} and \eqref{t12} are invariant under the flow \eqref{S} and this implies that $T^\ast=\infty$,
which concludes the proof.

\subsubsection{Blow-up }
To prove Corollary \ref{t2}, we use the same function $f(t)$ defined as \eqref{def}.
By the assumption \eqref{t13} with \eqref{poh} and \eqref{p}, we have
$$\|\sqrt{\mathcal K_\lambda} u_0\|^2>\|\sqrt{\mathcal K_\lambda}\varphi\|^2=t_1.$$
Thus, the continuity in time with \eqref{xx} gives
$$ \|\sqrt{\mathcal K_\lambda} u (t)\|^2>t_1,\quad \forall\, t\in [0,T^*) .$$
Hence, $\mathcal{MG}[u(t)]>1$ on $[0,T^*)$, and this and \eqref{t11} are invariant under flow \eqref{S}.

Finally, by using $\mathcal{E}[u(t)]>1$, $\mathcal{MG}[u(t)]>1$ and the identity $p\mathcal E[\varphi]=(p-1)\|\sqrt{\mathcal K_\lambda}\varphi\|^2$, we obtain for all $t\in[0,T^\ast)$
\begin{align*}
	\mathcal I[u(t)]&=\|\sqrt{\mathcal K_\lambda} u\|^2-\mathcal P[u]\\
	&=p\mathcal E[u]-(p-1)\|\sqrt{\mathcal K_\lambda} u\|^2\\
	&< p\mathcal E[\varphi]-(p-1)\|\sqrt{\mathcal K_\lambda}\varphi\|^2<0, 
\end{align*}
which concludes the proof by using Theorem \ref{t1}.

\section{Appendix: Morawetz estimate}
In this section, we present a virial identity (Proposition \eqref{mrwz}) that exhibits the convexity property in time for certain quantities associated with solutions of the generalized Hartree equation \eqref{S}. This identity serves as the basis for studying blow-up phenomena. The virial identity for the free nonlinear Schr\"odinger equation was first established by Zakharov \cite{Za} and Glassey \cite{G}. When the free equation is perturbed by an electromagnetic potential, Fanelly and Vega \cite{FV} derived the corresponding virial identities for the linear Schr\"odinger and linear wave equations. The proof relies on the standard technique of Morawetz multipliers, introduced in \cite{M} for the Klein-Gordon equation. The identity we present here is the same as that in \cite{sx}, with the addition of a term corresponding to the contribution from the inverse square potential.

\begin{proof}[Proof of Proposition \eqref{mrwz}]
Let $u \in C_t([0,T];H_\lambda^1)$ be a solution to the focusing case of equation \eqref{S}
\begin{align}\label{S1}
	i\partial_t u  
	&=-\Delta u + \frac\lambda{|x|^2}u -|x|^{-\tau}|u|^{p-2}\Big(I_\alpha*|\cdot|^{-\tau}|u|^p\Big)u \\
\nonumber
	&=-\Delta u + \frac\lambda{|x|^2}u -\mathcal N.
\end{align}
By multiplying $2 \bar u$ to \eqref{S1}, we obtain 
\begin{equation*}
-2 \Im{(\bar u \Delta u)} = \partial_t(|u|^2)
\end{equation*}
Using this, we can compute
$$V_\phi'(t)=2\Im\int_{\R^n} \bar u \nabla\phi\cdot\nabla u dx=2\sum_{k=1}^N \Im\int_{\R^n} \bar u \partial_k\phi\cdot \partial_k u dx.$$

In order to consider the second derivative of $V_\xi$, we need to compute 
\begin{equation}\label{sec}
V''_{\phi}(t)=2\sum_{k=1}^N \int_{\R^n}  \partial_k\phi\cdot \partial_t \Im(\bar u \partial_k u) dx.
\end{equation}
Using \eqref{S1}, we have  
\begin{align}
\partial_t\Im(\bar u \partial_k u )
&=\Re(i\partial_t u \partial_k\bar u )-\Re(i\bar{u} \partial_k \partial_t u )\nonumber\\
&=\Re\big(\partial_k\bar u (-\Delta u + \frac\lambda{|x|^2}u -\mathcal N)\big)-\Re\big(\bar u \partial_k(-\Delta u + \frac\lambda{|x|^2}u -\mathcal N)\big)\nonumber\\
&=\Re\big(\bar u \partial_k\Delta u -\Delta u\partial_k\bar u \big)+\Re\big(\bar u \partial_k\mathcal N-\mathcal N\partial_k\bar u \big)+\lambda\Re\big(\frac{u}{|x|^2}\partial_k\bar u-\bar u \partial_k(\frac{u }{|x|^2}) \big).\label{vr}
\end{align}
Here, for the last term, we see
\begin{equation}\label{aa}
	\Re\big(\bar u \partial_k(\frac{u }{|x|^2})-\frac{u}{|x|^2}\partial_k\bar u  \big)=-2\frac{x_k}{|x|^4}{|u|^2}.
\end{equation}
For the first two term, we will apply the following lemma, omitted here, can be found in the proof of \cite[Proposition 2.12]{sx}.
\begin{lem}\label{lem}
Let $\phi:\mathbb{R}^N \rightarrow \mathbb{R}$ be a radial, real-valued multiplier with $\phi=\phi(|x|)$. 
Then, for $\mathcal N$ defined as $\mathcal N = -|x|^{-\tau}|u|^{p-2}\big(J_\alpha*|\cdot|^{-\tau}|u|^p\big)u$, we have
\begin{align*}
&\Re\int_{\mathbb{R}^N}(\bar u \partial_k\Delta u -\Delta u\partial_k\bar u )+(\bar u \partial_k\mathcal N-\mathcal N\partial_k\bar u ) dx\\
&\qquad\qquad =\sum_{l=1}^{N}2\int_{\R^n}\partial_l\partial_k\phi\,\Re(\partial_ku\partial_l\bar u)dx-\frac12\int_{\R^n}\Delta^2\phi|u|^2dx\\
&\qquad \qquad\qquad-\frac{(p-2)}{p}\int_{\R^n}\Delta\phi|x|^{-\tau}|u|^p(J_\alpha*|\cdot|^{-\tau}|u|^{p})dx\\
&\qquad\qquad  \qquad \qquad -\frac{2\tau}p\int_{\R^n}x\cdot\nabla\phi|x|^{-\tau-2}|u|^{p}(J_\alpha*|\cdot|^{-\tau}|u|^p)dx\\
& \qquad \qquad\qquad  \qquad \qquad -\frac{2(N-\alpha)}p\int_{\R^n}\partial_k \phi|x|^{-\tau}|u|^{p}(\frac{x_k}{|\cdot|^2}J_\alpha*|\cdot|^{-\tau}|u|^p)dx.
\end{align*}
\end{lem}
Therefore, by combining \eqref{sec}, \eqref{vr}, Lemma \ref{lem} and \eqref{aa}, we finish the proof.
\end{proof}

\section{Declarations}
$\!\!\!\!\!\!\bullet$ The authors have no relevant financial or non-financial interests to disclose.\\
$\bullet$ The authors have no competing interests to declare that are relevant to the content of this article.\\
$\bullet$ All authors certify that they have no affiliations with or involvement in any organization or entity with any financial interest or non-financial interest in the subject matter or materials discussed in this manuscript.\\
$\bullet$ The authors have no financial or proprietary interests in any material discussed in this article.\\
$\bullet$ The data that support the findings of this study are available from the corresponding author
upon reasonable request.


\end{document}